\documentclass{amsart}
\usepackage{amssymb}
\usepackage{comment}

\newcommand{\e}{\varepsilon}

\newcommand{\Q}{\mathbb{Q}}

\newcommand{\R}{\mathbb{R}}

\newcommand{\Pro}{\mathbb{P}}

\newcommand{\Al}{\mathfrak{A}}

\newcommand{\norm}[1]{\|#1\|}

\newcommand{\dmu}{\mbox{ d}\mu}

\newcommand{\diam}{\text{diam}}
\newcommand{\RUC}{\ell\text{-RUC}(G)}

\usepackage{graphicx}

\newtheorem{theorem}{Theorem}[section]
\newtheorem{lemma}[theorem]{Lemma}

\newtheorem*{claim}{Claim}

\theoremstyle{definition}
\newtheorem{definition}[theorem]{Definition}
\newtheorem{example}[theorem]{Example}

\theoremstyle{remark}
\newtheorem{remark}[theorem]{Remark}

\numberwithin{equation}{section}

\begin{document}

\title{Uniform Continuity in Definable Groups}

\author{Alf Onshuus}
\address{Department of Mathematics, Universidad de los Andes, Bogot\'a, Colombia, Cra 1 No 18A-10}
% \curraddr{Department of Mathematics and Statistics,
% Case Western Reserve University, Cleveland, Ohio 43403}
\email{aonshuus@uniandes.edu.co}
% \thanks{The first author was supported in part by NSF Grant \#000000.}

\author{Luis Carlos Su\'arez}
\address{Department of Mathematics, Universidad de los Andes, Bogot\'a, Colombia, Cra 1 No 18A-10}
\email{lc.suarez2262@uniandes.edu.co}
% \thanks{Support information for the second author.}

%    General info
% \subjclass[2000]{Primary 54C40, 14E20; Secondary 46E25, 20C20}

% \date{January 1, 2001 and, in revised form, June 22, 2001.}

% \dedicatory{This paper is dedicated to our advisors.}

% \keywords{Model Theory, Amenability}

\begin{abstract}
In this paper we study analogues of amenability for topological groups in the context of definable structures. We prove fixed point theorems for such groups. More importantly, we propose definitions for definable actions and continuous functions from definable groups to topological spaces which might prove useful in other contexts.
\end{abstract}

\maketitle

\section{Introduction}

In this paper we propose a definable version for the concept of topological amenable groups, and prove some analogues of fixed point theorems. This provides some evidence that the notion of $\sigma$-continuity (see Definition \ref{Def:Sigma}) can be quite useful when understanding relations between definable and topological structures.

In order to understand our main results, we begin with the relevant definitions.

\begin{definition}\label{Def:Sigma}
Let $X$ be a set, a family $\Sigma \subseteq \Pro(X)$ is called a \emph{$\sigma$-topology} if
\begin{enumerate}
    \item $\emptyset, X \in \Sigma$.
     \item $\Sigma$ is closed under finite intersections.
    \item $\Sigma$ is closed under \emph{countable} unions.
   
\end{enumerate}
Every set in $\Sigma$ is called \emph{$\sigma$-open} and the complement of a $\sigma$-open set is $\sigma$-closed.  If $X$ and $Y$ are a $\sigma$-topological spaces, a function $f : X \rightarrow Y$ is said to be \emph{$\sigma$-continuous} if preimage of open ($\sigma$-open) set is $\sigma$-open.
\end{definition}

With this definition we define the following.

\begin{definition}\label{definition:definition2}
Let $G$ be a definable group. 
\begin{itemize}
    \item A function from $G$ to $\mathbb R$ will be a \emph{product $\sigma$-continuous function} if the map from $G\times G$ to $\mathbb R$ is $\sigma$-continuous.
    
    \item If $X$ is a $\sigma$-topological space, a \emph{$\sigma$-continuous action} from $G$ to $X$ is a group action such that the map $G\times X$ to $X$ is $\sigma$-continuous.
\end{itemize}
\end{definition}

In the case of an $\aleph_1$-saturated group, it can be shown that the following statements are equivalent:
\begin{enumerate}
    \item Every $\sigma$-continuous affine action of $G$ on the $\sigma$-topology generated by a basis $\Al(X)$, the $\sigma$-algebra generated by the elements of the chosen basis, of a non-empty compact convex set in a locally convex vector space has a fixed point.
    \item $G$ admits a left-invariant mean over the product $\sigma$-continuous functions in $G$.
    \item For every non-empty compact Hausdorff space $X$ with a basis $\Al(X)$, and every $\sigma$-continuous action of on $X$ (with the $\sigma$-topology generated by $\Al(X)$), there is a $G$-invariant probability measure over  the $\sigma$-algebra generated by $\Al(X)$.
\end{enumerate}
However, it can be shown that provided that $G$ is $\aleph_1$-saturated, $K$ is compact, and $\cdot : G \times K \rightarrow K$ is a \emph{definable action} (see, for example, \cite{KrHrPi}), since for each $g$ the map $K \rightarrow K$, $x \mapsto g \cdot x$ is an homeomorphism, it follows, from results showed in \cite{KrHrPi}, that the action $\cdot$ has a fixed point. 

Having in mind the topological situation, specifically the case when $G$ is not compact or locally compact, our motivation was to try to find a `good' definition of right uniform continuity in definable groups in order to find an analogue characterization of Amenability in terms of fixed points, provided that $G$ is not $\aleph_1$-saturated. We propose the following definition:
\begin{definition} Let $G$ be a definable group and $f : G \rightarrow \R^d$ be a function, $f$ is said to be \emph{logically right uniformly continuous} if for each $\e>0$ there are $(\Theta_n^{\e})_{n<\omega}$ definable sets in $G$ such that $G \subseteq \bigcup_{n<\omega} \Theta_n^{\e}$ and if $g_1,g_2 \in \Theta_n^{\e}$ for some $n<\omega$, then $|f(g_1h)-f(g_2h)|<\e$, for each $h \in G$.
\end{definition}

It can be shown that, in $\aleph_1$-saturated groups, product $\sigma$-continuous functions satisfy the previous definition and therefore, in this case the logically right uniformly continuous functions are a subset of the product $\sigma$-continuous functions.

The main results of this paper are the following:

\begin{itemize}
    \item Let $\text{Mean}(G)$ be the space of means over logically right uniformly continuous functions from $G$ to $\mathbb R$, endowed with the $\sigma$-topology generated by the basic open sets of the weak-* topology
    \[\mathcal U_{f,V}:=\{F\in \text{Mean}(G) | F(f)\in V\}\] where $f$ is a product $\sigma$-continuous function from $G$ to $\mathbb R$ and $V$ is an open subset of $[0,1]$. 
    
    Then the natural action of $G$ on $\text{Mean}(G)$ is $\sigma$-continuous.
    
    \item If $G$ is a definable group of a first order theory over a countable language, then there is a $G$-invariant mean on $\text{Mean}(G)$ if and only if any $\sigma$-continuous action from $G$ into a $\sigma$-topology generated by a basis $\mathfrak{B}$ of a compact Hausdorff topological space $X$ admits a $G$-invariant measure on the $\sigma$-algebra generated by $\Al(X)$.
    \end{itemize}

\section{$\sigma$-topology}\label{SectionTopology}
In this section we will discuss the $\sigma$-topological spaces and some of its properties.

We will define convergence of nets and separation conditions as $T_1$, Hausdorff, Lindel\"{o}f for $\sigma$-topologies in an analogous way to the usual topological definitions.

\begin{remark}
Limits of convergent nets in Hausdorff $\sigma$-topological spaces are unique.
\end{remark}
\begin{proof}
Suppose that is not the case and let $x \neq y$ be different limits of the net. As the $\sigma$-topology of $X$ is Hausdorff, there are disjoint $\sigma$-open set $U, V$ containing $x$ and $y$ respectively. As $(x_\alpha)_{\alpha \in A}$ is a convergent net there are $\alpha_1, \alpha_2 \in A$ such that $x_\alpha \in U$ for every $\alpha \geq \alpha_1$ and $x_\alpha \in V$ for every $\alpha \geq \alpha_2$. As $A$ is a directed set, let $\alpha_3 \in A$ such that $\alpha_1 \leq \alpha_3$ and $\alpha_2 \leq \alpha_3$. Then for every $\alpha \geq \alpha_3$, $x_\alpha \in U \cap V = \emptyset$, a contradiction.
\end{proof}
\begin{example} Let $X$ be a set.
\begin{enumerate}
    \item Every topology over $X$ is a $\sigma$-topology.
    \item Every $\sigma$-algebra over $X$ is a $\sigma$-topology and therefore every measurable space is also a $\sigma$-space.
    \item If $\mathcal{F}$ is a family of subsets of $X$ and we denote $\Sigma(\mathcal{F})$ as the smallest $\sigma$-topology over $X$ containing $\mathcal{F}$, then $(X, \Sigma(\mathcal{F}))$ is a $\sigma$-topological space.
\end{enumerate}
\end{example}
\begin{definition}
Let $(X, \Sigma)$ be a $\sigma$-space. A set $K \subseteq X$ is said to be \emph{countably compact} if every countable covering by $\sigma$-open sets admits a finite subcovering.
\end{definition}
\begin{example}
Let $X$ be a set.
\begin{enumerate}
    \item Assume that there is a compact topology $\tau$ over $X$, then $X$ is countably compact.
    \item Assume that $X$ is an first order structure for some language $L$, $X$ can always be seen as a $\sigma$-topological space when is endowed with the $\sigma$-topology generated by its definable sets. With this $\sigma$-topology one has that $X$ is countably compact if and only if $X$ is $\aleph_1$-saturated.
\end{enumerate}
\end{example}
\begin{definition}
Let $X$ be a $\sigma$-topological space, $X$ is said to be \emph{$\sigma$-normal} if every pair of disjoint $\sigma$-closed sets can be separated by $\sigma$-open sets.
\end{definition}
\begin{theorem}[Urysohn's Lemma for $\sigma$-spaces]\label{theorem:uryshon} Let $X$ be a $\sigma$-normal space; let $A$ and $B$ be disjoint $\sigma$-closed subsets of $X$. Then there exists a $\sigma$-continuous function $f: X \rightarrow [0,1]$ such that $f(x)=0$ for every $x \in A$, and $f(x)=1$ for every $x \in B$.
\end{theorem}\begin{proof}
The proof is similar to the standard proof of Urysohn's Lemma, so we will look into the details which differ from the proof in, for example, \cite{Munk}.

Let $P$ be an enumeration of the rational numbers in $[0,1]$. Let $A, B$ be disjoint $\sigma$-closed subsets of $X$, as $X$ is $\sigma$-normal there are disjoint $\sigma$-open sets $U_0, V_0$ such that $A \subseteq U_0$ and $B \subseteq V_0$. As $B$ is $\sigma$-closed, the set $U_1:=B^c$ is $\sigma$-open. Note that $A \subseteq U_0$ and $U_0 \subseteq V_0^c \subseteq U_1$. 

As in the proof of Urysohn's Lemma, we will define $U_p$ and $V_p$ for every rational $p,q$ in $(0,1)$ with $p<q$, there are $\sigma$-open sets $U_p, V_p, U_q, V_q$ in $X$ such that $U_p, V_p$ and $U_q$, $V_q$ are disjoint and whenever $p<q$, we have $U_p \subseteq V_p^c \subseteq U_q$.  Fix an enumeration of the rational numbers in $(0,1)$, and let $p$ is the $n+1$-th rational, we will construct sets $U_p$ and $V_p$ inductively as follows. Let $P_n$ be the rational numbers we have listed until the $n$-th step, adding $0,1$ so that they are always contained. Let $q_0$ and $q_1$ be the predecessor and the successor of $p$ in $P_n$.

Since $V_{q_1}^c \subseteq U_{q_2}$, by $\sigma$-normality, the disjoint $\sigma$-closed sets $V_{q_1}^c$ and $U_{q_2}^c$ can be separated. So, let $U_p, V_p$ be disjoint $\sigma$-open sets such that $V_{q_1}^c \subseteq U_p$ and $U_{q_2} \subseteq V_p$. We define $U_p$ to be the whole $X$ if $p>1$ and 
as $\emptyset$ if $p<0$.

Now, for every $x \in X$, define
\begin{align}
    \Q(x):=\{p \in \Q : x \in U_p\}.
\end{align}
Note that, since for every $p<0$, $U_p = \emptyset$ and for every $p>1$, $U_p=X$, the following function is well defined:
\begin{align*}
    & f : X \rightarrow [0,1], \\ & f(x):=\inf \Q(x) = \inf\{ p \in \Q : x \in U_p \}.
\end{align*}

 By definition, $f(x)=0$ if $x \in A$, and $f(x)=1$ for $x\in B$. 
 
Notice that (1) if $x \in V_r^c$, then $f(x) \leq r$ and (2) if $x \notin U_r$, then $f(x) \geq r$. If $x \in V_r^c$, then $x \in U_s$ for each $s>r$ by construction of these sets, so $\Q(x)$ contains all rational numbers greater than $r$, so $f(x)=\inf \Q(x) \leq r$. On the other hand, if $x \notin U_r$, then $x \notin U_s$ for any $s<r$. Therefore, $\Q(x)$ contains no rational numbers less than $r$, so that $f(x)=\inf \Q(x) \geq r$.

Now, let's prove the $\sigma$-continuity of $f$. Consider the following basic open sets of $[0,1]$: $[0,a), (a,b), (b,1]$. For the middle set, let $x_0 \in X$ be and take rational numbers $p<q$ such that
\begin{align*}
    a<p<f(x_0)<q<b.
\end{align*}
We want to see that $U_{p,q}:=U_q \setminus V_p^c$ is a $\sigma$-open neighborhood of $x_0$ contained in $f^{-1}((a,b))$. As $f(x_0)<q_1$, by condition (2) in the previous paragraph we know that $x_0 \in U_q$, while the fact that $f(x_0)>p$ implies by (1) that $x_0 \notin V_p^c$, thus $x_0 \in U_q \setminus V_p^c$.

Let $x \in U_{p,q}$. Since $x \in U_q \subseteq V_q^c$ we have $f(x) \leq q$, by (1). On the other hand, $x \notin V_p^c$ so that $f(x) \geq p$ by (2). Thus $f(x) \in [p,q] \subseteq (a,b)$, so $U_{p,q} \subseteq f^{-1}((a,b))$.  By taking the union over all rational numbers $p<q$ in $(a,b)$, we get that $f^{-1}((a,b))=\bigcup_{p,q} U_{p<q}$, for $p,q \in (a,b) \cap \Q$, is $\sigma$-open. The proof for the basic sets $[0,a)$ and $(a,1]$ is analogue. Therefore, $f$ is $\sigma$-continuous.
\end{proof}

As we have Urysohn's Lemma for $\sigma$-spaces, we can now show an analogue version of Riesz Representation Theorem.
\begin{definition}
Let $X$ be a $\sigma$-space, we define
\begin{align}
    C_\sigma(X):=\{ f : X \rightarrow \R : f \text{ is $\sigma$-continuous}\}.
\end{align}
It is easy to see that if $X$ is countably compact, then every $f \in C_\sigma(X)$ is bounded. In fact, in \cite{FixpunktAlf} it is shown that $C_\sigma(X)$ is a Banach space, provided that $X$ is countably compact.
\end{definition}

\begin{remark}
Let $X$ be a countably compact $\sigma$-space and let $\Al(X)$ be the $\sigma$-algebra generated by the $\sigma$-open sets of $X$. If $I : C_\sigma(X) \rightarrow \R$ is a bounded\footnote{Bounded with respect to the norm which makes $C_\sigma(X)$ a Banach space.} positive linear functional, for every $\sigma$-closed set $K \subseteq X$ we define
\begin{align}\label{definition:measure}
    \mu_0(K):=\inf\{I(f) : f \in C_\sigma(X), f \geq \chi_K\}.
\end{align}
We want to see that $\mu_0$ can be extended to a measure over $\Al(X)$.
\end{remark}
\begin{remark}
Let $X$ be a $\sigma$-space and $K \subseteq X$ be countably compact. If $C \subseteq K$ is $\sigma$-closed, then $C$ is also countably compact. Indeed, let $(U_n)_{n<\omega}$ be a countable $\sigma$-open covering of $C$. As $C$ is $\sigma$-closed, $C^c$ is $\sigma$-open and therefore $(U_n \cup C^c)_{n<\omega}$ is a countable $\sigma$-open covering of $K$. As $K$ is countably compact, there are $U_1, \dots, U_n, C^c$ which covers $K$ and as clearly $C^c$ is disjoint from $C$ we conclude that $C \subseteq U_1 \cup \dots \cup U_n$, i.e, $C$ is countably compact.
\end{remark}
\begin{lemma}\label{lemma:lindelofcountablyisnormal}
Let $X$ be a countably compact, Hausdorff, Lindel\"{o}f space. Then $X$ is $\sigma$-normal.
\end{lemma}
\begin{proof}
The proof is the same as for topological space by using the Lindel\"{o}f condition: as every open covering of $X$ admits a countable subcovering, the result follows from the countably compactness of $X$.
\end{proof}
\begin{lemma}\label{lemma:finitemeasure}
Let $X$ be a countably compact, Lindel\"{o}f, Hausdorff $\sigma$-space, $I:C_\sigma(X) \rightarrow \R$ be a bounded positive linear functional, and $\mu_0$ as in \eqref{definition:measure}. Then:
\begin{enumerate}
    \item For every $K, L$ $\sigma$-closed sets in $X$:
    \begin{enumerate}
        \item[(K1)] If $K \subseteq L$, then $0 \leq \mu_0(K) \leq \mu_0(L)<\infty$.
        \item[(K2)] $\mu_0(K \cup L) \leq \mu_0(K) + \mu_0(L)$.
        \item[(K3)] If $K$ and $L$ are disjoint, then the inequality in (K2) becomes an equality.
    \end{enumerate}
    \item If $K$ is $\sigma$-closed and $\e>0$, then there is a $\sigma$-open set $U$ in $X$ such that for every $\sigma$-closed set $L \subseteq U$, $\mu_0(L) \leq \mu_0(K)+\e$.
\end{enumerate}
\end{lemma}
\begin{proof}
(K1) is easy since $I$ is positive and clearly $K \subseteq L$ implies that $\chi_K \leq \chi_L$. Thus, $0 \leq \mu_0(K) \leq \mu_0(L)$. Since $I$ is bounded and positive, $|I(1)| \leq \norm{I}$. As clearly $\chi_X$ is the constant $\sigma$-continuous function $1$, $\mu_0(X) = I(1) \leq \norm{I}<\infty$. The result for the countably compact set $L$ follows from the fact that $X$ is countably compact, $L \subseteq X$ and the previous lines.

For (K2) if we take $f,g \in C_\sigma(X)$ such that $f \geq \chi_K, g \geq \chi_L$, then $f+g \in C_\sigma(X)$, as $\sigma$-continuous functions forms a vector space, and $f+g \geq \chi_K+\chi_L\geq \chi_{K \cup L}$. Then $\mu_0(K \cup L) \leq I(f+g)=I(f)+I(g)$. If we take the infimum over all $f,g \in C_\sigma(X)$ satisfying the condition $f \geq \chi_K$, $g \geq \chi_L$ we conclude that $\mu_0(K \cup L) \leq \mu_0(K)+\mu_0(L)$.

For (K3) it only remains to show that $\mu_0(K\cup L) \geq \mu_0(K)+\mu_0(L)$ provided that $K \cap L=\emptyset$. Let $h \in C_\sigma(X)$ such that $h \geq \chi_{K \cup L}$. Now, choose $\varphi \in C_\sigma(X)$ such that $\varphi|_K=1, \varphi|_L=0$. The existence of this function is given by Urysohn's Lemma for $\sigma$-spaces, as every countably compact, Lindel\"{o}f, Hausdorff space is $\sigma$-normal (see Lemma \ref{lemma:lindelofcountablyisnormal}) and the fact that $K$ and $L$ are closed disjoint subsets of the $\sigma$-normal space $X$. Set $f_h:=\varphi h, g_h:=(1-\varphi)h$. Then $f_h, g_h \in C_\sigma(X)$, $f_h \geq \chi_K$, $g_h \geq \chi_L$, and $f_h+g_h=h$. Note that
\begin{align*}
    \mu_0(K \cup L)= & \inf\{I(h) : h \in C_\sigma(X), h \geq \chi_{K \cup L}\} \\
    & = \inf\{I(f_h)+I(g_h) : h \in C_\sigma(X), h \geq \chi_{K \cup L}\} \\
    & \geq \inf\{ I(f) : f \in C_\sigma(X), f \geq \chi_K \} + \inf\{ I(g) : g \in C_\sigma(X), g \geq \chi_K \} \\
    & \geq \mu_0(K)+\mu_0(L).
\end{align*}
Finally, for (2), let $K$ be $\sigma$-closed and $\e>0$. Fix $\delta>0$, then by definition of $\mu_0$ there is $f \in C_\sigma(X)$ such that $f \geq \chi_K$ and $\mu_0(K)+\delta \geq I(f)$. Since $f$ is $\sigma$-continuous the set $U_\delta:=f^{-1}\left(\frac{1}{1+\delta},\infty \right)=\left\{ x \in X : f(x) > \frac{1}{1+\delta}\right\}$ is $\sigma$-open. Note that $K \subseteq U_\delta$, as if $x \in K$, then $f(x)\geq \chi_K(x)=1>\frac{1}{1+\delta}$.
Since $\mu_0(K)<\infty$, choose $\delta$ small enough such that $\delta(\mu_0(K)+\delta+1)<\e$. Then, for every $L \subseteq U_\delta$ $\sigma$-closed we have that $(1+\delta)f \geq \chi_L$, as $L \subseteq U_\delta$, $(1+\delta)f \in C_\sigma(X)$, and
\begin{align*}
    \mu_0(L)& = \inf\{I(h) : h \in C_\sigma(X), h \leq \chi_L\} \\
    & \leq (1+\delta) I(f) \leq (1+\delta)(\mu_0(K)+\delta)\\
    & < \mu_0(K)+\e.
\end{align*}
\end{proof}
\begin{lemma}\label{lemma:extracondition}
Let $X$ be a countably compact, Lindel\"{o}f, Hausdorff $\sigma$-space, denote $\mathfrak{C}$ as the set of all $\sigma$-closed sets in $X$, and let $\mu_0 : \mathfrak{C} \rightarrow [0, \infty)$ be such that (K1)-(K3) and (2) of Lemma \ref{lemma:finitemeasure} hold. Then,
\begin{enumerate}
    \item[(3)] For each $K, L \in \mathfrak{C}, K \subseteq L$,
    \begin{align*}
        \mu_0(L)-\mu_0(K)= \inf\{\mu_0(C) : C \in \mathfrak{C}, C \subseteq L \setminus K\}
    \end{align*}
\end{enumerate}
\end{lemma}
\begin{proof}
The result follows with the same arguments used in \cite{Elstrodt}, Lemma VIII.2.3, using the fact that $X$ is $\sigma$-normal.
% We shall prove both of these inequalities. So, let $C \in \mathfrak{C}$ such that $C \subseteq L \setminus K$. Then the disjoint union $C \cup K \subseteq L$ and
% \begin{align*}
%     \mu_0(L) \underbrace{\geq}_{\text{(K1)}} \mu_0(K \cup C) \underbrace{=}_{\text{(K3)}} \mu_0(K)+\mu_0(C).
% \end{align*}
% Taking the infimum over all such $C$ shows $\mu_0(L)-\mu_0(K)=\inf\{\mu_0(C) : C \in \mathfrak{C}, C \subseteq L \setminus K\}$.
% On the other hand, let $\e>0$ be. We will show that there is $C \subseteq L \setminus K$ $\sigma$-closed such that $\mu_0(C)+\e \geq \mu_0(L)-\mu_0(K)$, by the definition of the infimum of a set this will be enough. Choose $U$ $\sigma$-open, $K \subseteq U$ such that for every $H \in \mathcal{C}$, $H \subseteq U$, $\mu_0(H) \leq \mu_0(K)+\e$; there is such $U$ by condition (2).
% Since $L \setminus U$, $K$ are $\sigma$-closed and $X$ is $\sigma$-normal there are disjoint $\sigma$-open sets $W,V$ such that $K \subseteq V$ and $L \setminus U \subseteq W$. Let $C:=L \setminus V, D:=L \setminus W$ be and note that both of them are $\sigma$-closed. Besides, $C \subseteq L \setminus K$ and $\mu_0(D) \leq \mu_0(K)+\e$ because $D \subseteq U$. Since $C \cup D = (L \setminus V) \cup (L \setminus W) = L \setminus (V \cap W) = L$, because $W,V$ are disjoint, then
% \begin{align*}
%     \mu_0(L) \leq \mu_0(C)+\mu_0(D) \leq \mu_0(C)+\mu_0(K)+\e,
% \end{align*}
% which means that $\mu_0(L)-\mu_0(K) \leq \mu_0(C)+\e$, as we required.
\end{proof}
With Lemmas \ref{lemma:finitemeasure}, \ref{lemma:extracondition} now it is possible to show the following result:
\begin{theorem}[Extension Theorem]\label{theorem:extensionTheorem}
Let $X$ be a countably compact, Lindel\"{o}f, Hausdorff $\sigma$-space, $\mu_0 : \mathfrak{C} \rightarrow [0, \infty)$ such that (K1)-(K3) and (3) hold. Then there exists a unique measure $\mu : \Al(X) \rightarrow [0, \infty)$ such that $\mu|_{\mathfrak{C}}=\mu_0$ and for every $A \in \Al(X)$,
\begin{align}
    \mu(A)=\sup\{\mu(K) : K \subseteq A, K \in \mathfrak{C}\}.
\end{align}
\end{theorem}
\begin{proof}
The proof of this result is similar to the Carath\'eodory's construction of the outer measure. It can be found, for example in \cite{Elstrodt}; Satz VIII.2.4, and it only uses measure theoretic arguments, provided that we have a set function $\mu_0$ satisfying the requirements listed in the statement of the Theorem. The $\sigma$-inner regularity of $\mu$ arises from (3) in Lemma \ref{lemma:extracondition} and the uniqueness arises from the fact that $\mu_0$ can be extended to an $\sigma$-inner regular measure over $\Al(X)$. Note that $\Al(X)$ is the $\sigma$-algebra generated by both $\sigma$-closed sets and $\sigma$-open sets, as the complement of a $\sigma$-open set is $\sigma$-closed.
\end{proof}
\begin{remark}\label{remark:sigmaouterregular}
The measure obtained in the Extension Theorem (Theorem \ref{theorem:extensionTheorem}) is also $\sigma$-outer regular in $\sigma$-closed sets.
\end{remark}
\begin{proof}
Let $C$ be $\sigma$-closed and take $U \supseteq C$ be $\sigma$-open. As $\mu$ is finite, $\mu(U)<\infty$. Let $\e>0$, by $\sigma$-inner regularity, there is a $\sigma$-closed set $L \subseteq U \setminus K$ such that $\mu(L) \geq \mu(U \setminus K)-\e$. Let $V:=U \setminus L$ be $\sigma$-open and note that
\begin{align*}
    \mu(V)=\mu(U)-\mu(L) \leq \mu(U)-\mu(U\setminus K)+\e = \mu(K)+\e.
\end{align*}
\end{proof}
\begin{theorem}[Riesz Representation Theorem for $\sigma$-spaces]\label{theorem:Riesz}
Let $X$ be a countably compact, Lindel\"{o}f, Hausdorff $\sigma$-space and $I: C_\sigma(X) \rightarrow \R$ be a bounded positive linear functional. Then there exists a unique measure $\mu : \Al(X) \rightarrow [0, \infty)$ such that
\begin{align}\label{eq:integralriesz}
    I(f) = \int_X f \text{ \emph{d}}\mu, \hspace{5mm} f \in C_\sigma(X).
\end{align}
Also, the following equalities hold:
\begin{align}
    \label{eq:Riesz1}& \forall K \subseteq X, K \in \mathfrak{C}, \hspace{5mm} \mu(K)=\inf\{I(f) : f \in C_\sigma(X), f \geq \chi_K\}, \\
    \label{eq:Riesz2}& \forall A \in \Al(X), \hspace{12mm}   \mu(A)=\sup\{\mu(K) : K \subseteq A, K \in \mathfrak{C}\}.
\end{align}
\end{theorem}
\begin{proof}
Uniqueness of $\mu$ follows from \eqref{eq:Riesz2} and \eqref{eq:Riesz2} follows from Extension Theorem. To see \eqref{eq:Riesz1}, let $K \in \mathfrak{C}$ and $f \in C_\sigma(X)$ such that $f \geq \chi_X$. Then, by \eqref{eq:integralriesz}:
\[I(f)=\int_X f\dmu \geq \int_X \chi_K \dmu = \mu(K).\]
Thus we get that $\mu(K)$ is a lower bound for $\{I(f) : f \in C_\sigma(X), f \geq \chi_K\}$. Let $K \in \mathfrak{C}$ and $\e>0$ be. Since $\mu$ is $\sigma$-outer regular on $\sigma$-closed sets, we have a $\sigma$-open set $U \supseteq K$ such that $\mu(U) \leq \mu(K)+\e$. By Urysohn's Lemma for $\sigma$-spaces there is $f \in C_\sigma(X)$ such that $f|_K=1$ and $f|_{U^c}=0$; thus $\chi_K \leq f \leq \chi_{U}$. The later inequality along with \eqref{eq:integralriesz} implies that
\[I(f) = \int_X f\dmu \leq \int_X \chi_U \dmu = \mu(U) \leq \mu(K)+\e.\]
Now we shall see that $I(f) \geq \int_X f \dmu$, for each $f \in C_\sigma(X), f \geq 0$. For $A \subseteq X$, define the closure of $A$, $\overline{A}$, as the intersection of all $\sigma$-closed sets containing $A$. Note that $\overline{A}$ is not necessarily either $\sigma$-open or $\sigma$-closed as the $\sigma$-topology is only closed under countable set operations.

It suffices to show the statement for simple positive functions $u$ such that $u \leq f$, because $\sigma$-continuous functions are measurable in $\Al(X)$. Let $u=\sum_{j \leq m} \alpha_j \chi_{A_j}$ be a simple positive function such that $\alpha_1, \dots, \alpha_m>0$, $A_1, \dots, A_m \in \Al(X)$ are pairwise disjoint and $u \leq f$. Let $\e>0$ with $\e<\min\{\alpha_1, \dots, \alpha_m\}$. By $\sigma$-inner regularity of $\mu$, for each $j=1, \dots, m$, there are $\sigma$-closed sets $K_j$ such that $\mu(A_j) \leq \mu(K_j) + \e$. Since the $K_j$ are pairwise disjoint there $U_j \supseteq K_j$ pairwise disjoint $\sigma$-open neighborhoods. Without loss of generality, we may assume that $U_j \subseteq \{x \in X : f(x)>\alpha_j-\e\}$ as the latter set is $\sigma$-open and contains $K_j$.

By using Urysohn's Lemma for $\sigma$-spaces, for each $j=1, \dots, m$ we get $\varphi_j \in C_\sigma(X)$ such that $\chi_{K_j} \leq \varphi_j \leq \chi_{U_j}$. Let $g:=\sum_{j\leq m}(\alpha_j-\e)\varphi_j \in C_\sigma(X)$ and note that $0 \leq g \leq f$. Since $I$ is positive, we get that:
\begin{align*}
    I(f) & \geq I(g) =\sum_{j \leq m} (\alpha_j-\e)I(\varphi_j) \\ & \geq \sum_{j \leq m} (\alpha_j-\e)\mu(K_j) \\ & \geq \sum_{j \leq m} (\alpha_j-\e)(\mu(A_j)-\e) \\ & = \int_X u \dmu - \e \underbrace{\sum_{j \leq m}(\mu(A_j)\underbrace{-\e+\alpha_j}_{\geq 0})}_{\geq 0} \\ & \geq \int_X u \dmu.
\end{align*}

Now, let's check the equality for positive functions in $C_\sigma(X)$, this concludes the proof. Let $f \in C_\sigma(X)$, without loss of generality, assume that $0\leq f \leq 1$ (both $I$ and $\int_X \cdot \dmu$ are linear). Let $\text{supp } f := \overline{f^{-1}(\R \setminus \{0\})}$ be. By definition of the closure, there is a $\sigma$-closed sets $K$ such that $\text{supp }f \subseteq K$.
Let $\e>0$ be, by $\sigma$-outer regularity of $\mu$ on $\sigma$-closed sets (see Remark \ref{remark:sigmaouterregular}), we take a $\sigma$-open set $V \supseteq K$ such that $\mu(V) \leq \mu(K)+\e$. Note that $K$ and $V^c$ are disjoint $\sigma$-closed sets, thus by $\sigma$-normality there are $U, W$ disjoint $\sigma$-open neighborhoods of $K$ and $V^c$ respectively. As $V^c \subseteq W$, $W^c \subseteq V$ and being $W$ $\sigma$-open implies that $C:=W^c$ is $\sigma$-closed. Note that the following chain of continences holds: $\text{supp }f \subseteq K \subseteq U \subseteq C \subseteq V$.

By Urysohn's Lemma for $\sigma$-spaces choose $\varphi \in C_\sigma(X)$ such that $\varphi|_{K}=1$ and $\varphi|_{U^c}=0$. Note that $\varphi-f \geq 0$ and $\varphi - f \in C_\sigma(X)$. As $I(f) \geq \int_X f \dmu$,
\begin{align*}
    I(\varphi)-I(f)=I(\varphi-f) \geq \int_X (\varphi - f)\dmu = \int_X \varphi \dmu - \int_X f \dmu \geq 0.
\end{align*}
Therefore,
\begin{align}\label{eq:inequality2}
    0 \leq I(f)-\int_X f \dmu \leq I(\varphi)-\int_X \varphi \dmu.
\end{align}
It is easy to see that $\text{supp }\varphi \subseteq C$ as $\varphi|_{U^c}=0$ and $U \subseteq C$. Since $\varphi|_K=1$, $\varphi \geq 0$, $\int_X \varphi \dmu \geq \mu(K)$. Since $\text{supp }\varphi \subseteq C$, we get $0 \leq \varphi \leq \chi_C$. Thus, for any $\sigma$-continuous function $h$ such that $\chi_C \leq h$, $0 \leq \varphi \leq \chi_C \leq h$. As $I$ is positive, $I(\varphi) \leq I(h)$ and by taking the infimum over all such $h$ we get that $I(\varphi) \leq \mu(C)$. Then the inequality in \eqref{eq:inequality2} can be estimated in the following way:
\begin{align*}
    0 \leq I(f)-\int_X f \dmu \leq I(\varphi)-\int_X \varphi \dmu < \mu(C) - \mu(K) \leq \mu(V) - \mu(K)\leq \e.
\end{align*}
As $\e>0$ was arbitrary, we conclude that $I(f)-\int_X f \dmu =0$
\end{proof}

\begin{definition}
Let $G$ be a definable group and $X$ be a topological space, and fix a basis $\mathcal B$ for the topology of $X$. An action $\cdot : G \times X \rightarrow X$ is said to be \emph{$\sigma$-continuous for $\mathcal B$} if the action is a $\sigma$-continuous function when we give $G\times X$ the product $\sigma$-topology generated by the products of definable subsets of $X$ with elements of $\mathcal B$.

In other words, an action is $\sigma$-continuous for $\mathcal B$ if the preimage of an element in $\mathcal B$ is a countable union of cartesian products of definable subsets of $G$ and elements of $\mathcal B$.
\end{definition}

Once we have fixed $\mathcal B$, and if there are no grounds for confusion, we may not make it explicit.

\section{Logically right uniformly continuous functions}
\begin{definition}\label{definition:definition1}Let $G$ be a definable group and $f : G \rightarrow \R^d$ be a function, $f$ is said to be \emph{logically right uniformly continuous} if for each $\e>0$ there are $(\Theta_n^{\e})_{n<\omega}$ definable sets in $G$ such that $G \subseteq \bigcup_{n<\omega} \Theta_n^{\e}$ and if $g_1,g_2 \in \Theta_n^{\e}$ for some $n<\omega$, then $|f(g_1h)-f(g_2h)|<\e$, for each $h \in G$.
\end{definition}

\begin{lemma}
The logically right uniformly continuous functions form a vector space.
\end{lemma}

\begin{proof}
Let $f, g$ be logically right uniformly continuous. Fix $\e>0$; since $f, g$ are logically right uniformly continuous, there are $\left(\Theta_i^{\frac{\e}{2},f}\right)_i, \left(\Theta_j^{\frac{\e}{2},g}\right)_j$ sequences of definable sets in $G$ such that $G \subseteq \bigcup_{i<\omega} \Theta_i^{\frac{\e}{2},f}, \text{ }\bigcup_{i<\omega} \Theta_i^{\frac{\e}{2},g}$ and if $x_1, x_2 \in \Theta_i^{\frac{\e}{2},f}, \Theta_j^{\frac{\e}{2},g}$, for each $h \in G$
\begin{align*}
    & |f(x_1h)-f(x_2h)|<\frac{\e}{2}, \\
    & |g(x_1h)-g(x_2h)|<\frac{\e}{2}.
\end{align*}
Define $\Theta_{i,j}^{\e, f+g}:= \bigcap_{i,j \text{  finite}}\Theta_i^{\frac{\e}{2},f}\cap \Theta_j^{\frac{\e}{2},g}$; clearly each $\Theta_{i,j}^{\e,f+g}$ is definable (finite intersection of definable sets). Now, let's check that $\bigcup_{i,j<\omega} \Theta_{i,j}^{\e,f+g}$ covers $G$. Let $g \in G$ be. Then there are $i,j<\omega$ such that $x \in \Theta_i^{\frac{\e}{2},f}\cap \Theta_j^{\frac{\e}{2},g}$. Thus $g$ is in some of these finite intersections and the countable union over all $i,j$ covers $G$. 

Take $x_1, x_2 \in \Theta_{i,j}^{\e,f+g}$. Then, $x_1, x_2$ are in the finite intersection of $\Theta_i^{\frac{\e}{2},f}, \Theta_j^{\frac{\e}{2},g}$. By definition of these sets, we have that for each $h \in G$, $|(f+g)(x_1h)-(f+g)(x_21h)|<\e$
\end{proof}

\begin{lemma}
The bounded logically right uniformly continuous functions form a Banach space.
\end{lemma}

\begin{proof}
Let $(f_n)_n$ be a sequence of bounded logically right uniformly continuous functions. Since $B(G)$ is a Banach space, there is a bounded function $f : G \rightarrow \R$ such that $\norm{f_n-f}_\infty \rightarrow 0$. Let $\e>0$ be. By uniform convergence, there is an $N<\omega$ such that $\norm{f_n - f}_\infty < \frac{\e}{2}$, for $n \geq N$. Since $f_N$ is logically right uniformly continuous, there is a sequence of defiable sets $(\Theta_j^{\frac{\e}{3},N})_j$ such that $G \subseteq \bigcup_{j<\omega}\Theta_j^{\frac{\e}{3},N}$ and if $x_1, x_2 \in \Theta_j^{\frac{\e}{3},N}$, then for each $h \in G$, $|f_N(x_1h)-f_N(x_2h)|<\frac{\e}{3}$. Take this definable covering of $G$ and let $x_1, x_2 \in \Theta_j^{\frac{\e}{3},N}$, for some $j<\omega$, and $h \in G$ then
\begin{align*}
    |f(x_1h)-f(x_2h)| \leq \underbrace{|f(x_1h)-f_N(x_1h)|}_{<\frac{\e}{3} \text{ by uniform convergence}}+\underbrace{|f_N(x_1h)-f_N(x_2h)|}_{<\frac{\e}{3}\text{ by choice of $N$}}+\underbrace{|f_N(x_2h)-f(x_2h)|}_{<\frac{\e}{3} \text{ by uniform convergence}}.
\end{align*}
Thus $f$ is logically right uniformly continuous, as we wanted to show.
\end{proof}

\begin{definition}
For $G$ a definable group, we define
\[\RUC:=\{f \in B(G) : f \text{ is bounded and logically right unformly continous}\}.\]

Let $\text{Mean}(G)$ be the set of means over $\RUC$. Since $\RUC$ is a Banach space, by Banach-Alaogulu's Theorem, $\text{Mean}(G)$ is compact in the weak*-topology of $(\RUC)'$.
\end{definition}

\begin{theorem}\label{theorem:continuityofaction}
If $G$ is a definable group, then the canonical action of $G$ into $\text{Mean}(G)$ is $\sigma$-continuous.
\end{theorem}
\begin{proof}
Let $U$ be any fixed open subset of $\mathbb R$, let $f \in \RUC$ be
fixed.

Let
\[
\mathcal U_{f,U}:= \{(h, F')\mid \ F'\in \mathcal F, h\in G, \
|F'(_h f)\in U\}
\]
be the basic open subset of $\mathcal F$ determined by $f$ and
$U$.

For each rational $\e$ fix a countable set of definable sets of $G$ given by definition logically right uniformly continuous functions, and let $\Xi$ be the set of all the
choices $\Theta^{\e}_i$ with $\e$ rational.

For each $\Theta^{\e}_i\in \Xi$, let $\{F_i(_{g_i}f)\}_{i\in
\omega}$ be a dense subset of
\[
\left\{ F'\left(_h f\right) \mid h\in \Theta_i^{\e}, \ F'\in
\mathcal F \right\},
\]
and let $\Upsilon^{\e}$ be the set of all open subsets of
$\mathcal F$ of the form
\[
\{F'\in \mathcal F \mid \ |F'(_{g_i} f)-F_i(_{g_i}f)|<q\}
\]
where $q$ varies over all rational numbers. Finally, let
$\Upsilon$ be the union of all the countable sets $\Upsilon^{\e}$
with rational $\e$. So $\Upsilon$ is countable.

\begin{claim}
It is enough to show that for any $F\in \mathcal F$ and any $g\in
G$ such that $_gF\in \mathcal U_{f,U}$, we can find some
$\Theta\in \Xi$ and some $\mathcal U\in \Upsilon$ such that $g\in
\Theta$, $F\in \mathcal U$ and $_hF'\in \mathcal U_{f,U}$ for all
$h\in \Theta$, $F'\in \mathcal U$.\end{claim}

\begin{proof}
It would follow from the hypothesis that
\[
\{(h, F')\mid _hF'\in \mathcal U_{f,U}\}
\]
is a union of sets of the form $\Theta\times \mathcal U$ in
$\Xi\times \Upsilon$. Since the latter is a countable set, we can
conclude that $\mathcal U_{f,U}$ is a countable union of the
cartesian product of definable sets and basic open subsets of
$\mathcal F$, as required.
\end{proof}

So let $F$ and $g$ be as in the statement of the previous claim.
Let $\epsilon$ be a rational number such that
$B_{2\e}(F(_gf))\subset U$.

Let $\Theta\in \Xi$ be a definable set such that $g\in \Theta$ and
$|f(g_1h)-f(g_2h)|<\e/2$ for any $g_1, g_2\in G$. It follows that
\begin{equation}\label{eqnAlf}
\forall F'\in \mathcal F, \left(\left(g_1, g_2\in G\right)\tag{*}
\Rightarrow
\left(\left|F'\left(_{g_1}f\right)-F'\left(_{g_2}f\right)\right|<\epsilon/2\right)\right).
\end{equation}

By density of $F_i, g_i$ in the construction of $\Upsilon$, let
$F_i, g_i$ be such that $g_i\in \Theta$ and
\[
|F_i(_{g_i}f)-F(_gf)|<\epsilon/2,
\]
and let $\mathcal W$ be the set in $\Upsilon$ defined by
\[
\{F'\in \mathcal F \mid \ |F'(_{g_i} f)-F_i(_{g_i}f)|<\epsilon\}.
\]

\bigskip

We will prove that $\Theta\times \mathcal W$ satisfies the
required conditions.

\medskip

First, notice that $$\begin{array}{llll}
|F(_gf)-F_i(_{g_i}f)|&<&\e/2 & \text{ by choice} \\
|F(_{g_i}f)-F(_gf)|&<&\e/2 & \text{ by \eqref{eqnAlf}}
\end{array}$$
so by triangular inequality
\[
|F(_{g_i}f)-F_i(_{g_i}f)|<\e
\]
so that $F\in \mathcal W$.

\medskip

Now, for any $h\in \Theta$ and $F'\in W$ we have that

$$\begin{array}{llll}
|F'(_hf)-F'(_{g_i}f)|&<&\e/2 &\text{ by \eqref{eqnAlf}} \\
|F'(_{g_i}f)-F_i(_{g_i}f)|&<&\e &\text{ by choice of $\mathcal W$}\\
|F'(_{h}f)-F_i(_{g_i}f)|&<&\e+\e/2 &\text{ by triangular
inequality}\\
|F_i(_{g_i}f)-F(_{g}f)|&<&\e/2 &\text{ by choice of $F_i$, $g_i$}\\
|F'(_{h}f)-F(_{g}f)|&<&2\e &\text{ by triangular inequality}.
\end{array}$$

So $F'(_hf)\in B_{2\e}(F(_gf)\subset U$, as required.
\end{proof}

\begin{theorem}\label{theorem:RUC}
Let $G$ be a definable group, $X$ be a compact space. Suppose that $G$ acts $\sigma$-continuously on $X$. Let $f : X \rightarrow \R$ be a $\sigma$-continuous function (with respect to the $\sigma$-topology which makes $\cdot$ $\sigma$-continuous). Then for each $\zeta \in X$ the mapping $g \mapsto f(g \cdot \zeta)$ is logically right uniformly continuous.
\end{theorem}
Fix $\e>0$. Let $(V_N)_N$ be a countable covering of $\{f(x\cdot \zeta) : x \in G, \zeta \in X\} \subseteq \R$ such that $\diam(V_N)<\e$. Since $f$ is $\sigma$-continuous and the action is also $\sigma$-continuous, for each $N<\omega$ 
\[\cdot^{-1}(f^{-1}(V_N))=\bigcup_{i<\omega} \Phi^N_i \times U^N_i\]
for $\Phi_i^N$ definable in $G$ and $U^N_i$ $\sigma$-open in $X$.

Clearly, for each $x \in G$ and for each $\zeta \in X$ there is $N<\omega$ such that $f(x\zeta) \in V_N$. Therefore, there exists $i<\omega$ such that $f(\Phi_i^N\cdot U_i^N) \subseteq V_n$. This in particular means that $\diam(f(\Phi_i^N\cdot U_i^N))<\e$. Therefore, $(x, \zeta) \in \Phi^{x,\zeta} \times U^{x, \zeta}$, where $\Phi^{x,\zeta} \times U^{x, \zeta}$ is one of these boxes (we know that at least there is one).

Note that $(U^{x,\zeta})_{\zeta \in X}$ is an open covering of $X$ (each $U^{x,\zeta}$ is the countable union of open basic sets in $X$ and therefore open). Being $X$ compact this means that there are $\zeta_1, \dots, \zeta_k$ such that $X \subseteq \bigcup_{i \leq k} U^{x,\zeta_i}$.

Associated to this finite sequence of $U^{x,\zeta_i}$ we have a finite family of definable sets $\Phi^{x,\zeta_i}$ all of them containing $x$. Note that: for each $i \leq k$, $\diam(f(\Phi^{x,\zeta_i}\cdot U^{x,\zeta_i}))<\e$. This implies that $\diam\left(f\left(\bigcap_{i \leq k} \Phi^{x,\zeta_i} \cdot X \right) \right)<\e$. Indeed, let $(y,\eta) \in \bigcap_{i \leq k} \Phi^{x,\zeta_i} \times X$. Then $y \in \bigcap_{i \leq k} \Phi^{x,\zeta_i}$ and there is $j=1, \dots, k$ such that $\eta \in U^{x,\zeta_j}$. In particular, $y \in \Phi^{x,\zeta_j}$ and therefore $(y,\eta) \in \Phi^{x,\zeta_i}\times U^{x,\zeta_j}$. The latter assertion implies that $\bigcap_{i \leq k} \Phi^{x,\zeta_i} \times X \subseteq \Phi^{x,\zeta_i}\times U^{x,\zeta_j}$ for some $j=1, \dots, k$ and therefore $\diam\left(f\left(\bigcap_{i \leq k} \Phi^{x,\zeta_i} \cdot X\right) \right)<\e$.

Define $\Theta_x:=\bigcap_{i\leq k} \Phi^{x,\zeta_i}$. Since the sets $\Phi^N_i \times U_i^N$ came from a covering of $\{f(x\cdot \zeta) : x \in G, \zeta \in X\}$, clearly $\bigcup \bigcap_{i\leq k} \Phi^{x,\zeta_i}=G$ and the latter set is the countable union of finite intersections; therefore we can cover $G$ with countable many $\Theta_x$.

Let $g_1, g_2 \in \Theta_x$, therefore, for each $\zeta \in X$, since $\diam\left(f(\Theta_x \cdot X)\right)<\e$, we have that $|f(g_1\zeta)-f(g_2\zeta)|<\e$.

\begin{theorem}\label{Main}
For a definable group the following statements are equivalent:
\begin{enumerate}
    \item Every $\sigma$-continuous affine action of $G$ on a non-empty compact convex set in a locally convex vector space has a fixed point.
    \item $G$ admits a left-invariant mean over $\RUC'$.
    \item For every non-empty compact Hausdorff space $X$ and every $\sigma$-continuous action of $G$ on $X$, there is a $G$-invariant probability measure over $\Al(X)$, the $\sigma$-algebra generated by a basis of the topology of $X$.
\end{enumerate}
\end{theorem}
\begin{proof}

(1) $\Rightarrow$ (2): The set Mean$(G)$ of all means on $\RUC$ is a weak*-closed subset of the unit ball in $\RUC'$, which is compact in the weak*-topology by Banach-Alaoglu's Theorem. By Theorem \ref{theorem:continuityofaction} we know that the canonical action of $G$ in Mean$(G)$ is $\sigma$-continuous and it is clear that is affine. Therefore, the action admits a fixed point; but a fixed point of such action must be a left-invariant mean over $\RUC$.

(2) $\Rightarrow$ (3): Assume that there is a left-invariant mean over $\RUC$. Now, let $(X, \tau)$ be a compact Hausdorff space and let $\rho$ be a $\sigma$-topology over $X$ generated by a basis of $\tau$ which witnesses a $\sigma$-continuous action $\cdots$ from $G$ to $X$. Being $\tau$ a compact, Hausdorff topology over $X$, it follows that $\rho$ is a countably compact, Hausdorff, Lindel\"{o}f $\sigma$-topology.

Fix $x_0\in X$, for $f \in C_\sigma(X)$, define a function
\begin{align}
    F_f : G \rightarrow \R, \hspace{5mm} F_f(g):=f(g \cdot x_0).
\end{align}
By Theorem \ref{theorem:RUC} we know that $F_f$ is logically right uniformly continuous. Being $\rho$ a countably compact $\sigma$-topology over $X$ and $f \in C_\sigma(X)$ it is clear that $F_f$ is bounded. Let $a \in G$. Then
\begin{align*}
    F_{_{a}f}(g)=\mbox{}_{a}f(g \cdot x_0)=f(a \cdot (g \cdot x_0))=f(ag \cdot x_0)=F_f(ag)=\mbox{}_{a}F_f(g)
\end{align*}
for all $g \in G$. Let $m$ be a left invariant mean on $\RUC$. Define the mapping,
\begin{align}
    \varphi : C_\sigma(X) \rightarrow \R, \hspace{5mm} \varphi(f):=m(F_f).
\end{align}
Let $f \in C_\sigma(X)$, $\norm{f} = 1$. Then 
\[|\varphi(f)|=|m(F_f)| \leq \norm{m} \norm{F_f} = \sup_{g \in G} |f(g \cdot x_0)|=1.\] 

By definition of a mean, \[\inf_{x \in G} h(x) \leq m(h) \leq \sup_{x \in G} h(x)\] for every $h \in \RUC$. So if $f \geq 0$ and $f \in C_\sigma(X)$, then $\varphi(f) = m(F_f) \geq \inf_{x \in G} f(x \cdot x_0) \geq 0$. Linearity follows from linearity of $m$.

As $\varphi$ is a positive, bounded linear functional over $C_\sigma(X)$ and $X$ is a countably compact, Hausdorff, Lindel\"{o}f $\sigma$-space, by Riesz Representation Theorem (Theorem \ref{theorem:Riesz}) there is a $\sigma$-inner regular probability measure $\mu$ over the $\sigma$-algebra generated by the $\sigma$-open sets of $X$ such that, for every $f \in C_\sigma(X)$,
\begin{align}\label{eq:representation}
    \varphi(f)=\int_X f\dmu.
\end{align}
Since $m$ is left invariant, we have that
\begin{align*}
    \varphi(_{g}f)=m(F_{_{g}f})=m(_{g}F_f)=m(F_f)=\varphi(f)
\end{align*}
for every $f \in C(X)$ and $g \in G$. By the integral representation of $\varphi$ given in \eqref{eq:representation}, we know that $\mu$ is $G$-invariant.

\bigskip

For (3) $\Rightarrow$ (1), we need the following claim.
\begin{claim}
Let $X, Y$ be Hausdorff $\sigma$-spaces and $f : X \rightarrow Y$ be $\sigma$-continuous. Then for every converging net $(x_\alpha)_\alpha \rightarrow x$ in $X$ we have that $f(x_\alpha) \rightarrow f(x)$ in $Y$.
\end{claim}
\begin{proof}
Let $(x_\alpha)_\alpha$ be a convergent net to a point $x \in X$. As $f$ is continuous, for every $\sigma$-open set $U$ containing $f(x)$, we have that $f^{-1}(U)$ is $\sigma$-open in $X$. As the net $(x_\alpha)_\alpha$ converges to $x$ and $f^{-1}(U)$ is an open set containing $x$, by convergence of $(x_\alpha)_\alpha$ we have that $x_\alpha$ is eventually in $f^{-1}(U)$. Therefore $f(x_\alpha)$ is eventually in $f(f^{-1}(U)) \subseteq U$. Thus $f(x_\alpha) \rightarrow f(x)$.
\end{proof}

Now (3) $\Rightarrow$ (1) follows the same proof of the analogous results in \cite{FixpunktAlf} or \cite{Wagon}. We refer to the reader to the proof of Theorem 3.5 in  \cite{FixpunktAlf}, although we will mention the changes that we need for the proof to work in the $\sigma$-continuous case. Assume (3) and let $\cdot$ be a $\sigma$-continuous action on a non-empty compact set in a locally convex vector space $L$. So, first we will take $F=\{U_1, \dots, U_n\}$ a minimal covering of $C$ by basic open sets in $L$. Define $S_i:=U_i \setminus \left(\bigcup_{j=1}^{i-1}U_j\right) \in \Al(X)$ and let $\mu_{U_i}:=\mu(S_i)$, for each $i=1, \dots, n$, so that in particular $\sum_{i \leq n}\mu_{U_i}=1$. As the elements in the sets $F \in D$ are basic open sets and therefore open sets in the topology of $L$, the Claims 3.6, 3.7, and 3.8 in \cite{FixpunktAlf} remain true. Finally, note that as the $\sigma$ topology of $C$ is contained in the topology of $C$, the convergence of the net in the topology of $C$ implies the convergence of the net in the $\sigma$-topology. As the sets in the elements of $D$ are basic open sets and $\mu$ is $G$-invariant, we get a fixed point using the continuity of $\cdot$ and using the previous claim.
\end{proof}

\bibliographystyle{amsplain}

\end{document}